\documentclass[reqno,11pt]{amsart}
\usepackage{epsfig,amscd,amssymb,amsmath,amsfonts}
\usepackage{amsmath}
\usepackage{graphicx}

\newtheorem{theorem}{Theorem}[section]
\newtheorem{lemma}[theorem]{Lemma}
\newtheorem{proposition}{Proposition}[section]
\theoremstyle{definition}
\newtheorem{definition}[theorem]{Definition}

\newtheorem{conjecture}[theorem]{Conjecture}

\newtheorem{example}[theorem]{Example}

\theoremstyle{remark}
\newtheorem{remark}[theorem]{Remark}

\numberwithin{equation}{section}

\usepackage[margin=1.05in]{geometry}
\usepackage[colorlinks]{hyperref}
\usepackage{siunitx}

\makeatletter
\let\@wraptoccontribs\wraptoccontribs
\makeatother

\begin{document}

\title[The Exterior Graded Swiss-Cheese Operad ${\Lambda}^{S^2}_V$]{The Exterior Graded Swiss-Cheese Operad ${\Lambda}^{S^2}_V$\\ {\scriptsize (with an appendix by Ana Lorena Gherman and Mihai D. Staic)} }

\author{Mihai D. Staic}
\address{
Department of Mathematics and Statistics, Bowling Green State University, Bowling Green, OH 43403 }
\address{Institute of Mathematics of the Romanian Academy, PO.BOX 1-764, RO-70700 Bu\-cha\-rest, Romania.}

\email{mstaic@bgsu.edu}





\begin{abstract} In this paper we present a construction which is a generalization of the exterior algebra of a vector space $V$. We show how this fits in the language of operads, discuss some properties, and  give explicit computations for the case $dim(V)=2$. 

\end{abstract}

\maketitle


%

\section*{Introduction}

Exterior algebra of a free module plays an important role in several areas of mathematics. From  determinants to  differential forms, from Hochschild homology to grassmannian, 
the exterior algebra pops up at the center of several seemingly different constructions. For an account of its connections with combinatorics, topology, Lie theory, mathematical physics, and algebraic geometry one can check   \cite{f}.  

The starting point of this paper was the connection between exterior algebra and  Hochschild homology. It is well known that Hochschild homology is modeled by the simplicial structure of the sphere $S^1$ (see \cite{p}). There is also the Hochschild-Kostant-Rosenberg theorem which gives a description for $HH_{\bullet}(A)$ in terms of the exterior algebra of module of differential $\Omega_{A/k}$ (see \cite{hkr}, \cite{lo}). So, in a certain way, one can think about the exterior algebra as being associated to the sphere $S^1$. This statement is a little bit misleading, since there is no explicit construction that starts with a simplicial structure of $S^1$ and a vector space $V$ as input, and produces the exterior algebra of $V$ as output. However this analogy was the starting point of our construction.    

In this paper we introduce a construction that mimics the exterior algebra, but instead of using the sphere $S^1$ as a model we use the sphere $S^2$. This was inspired by work of Pirashvili  on Higher Hochschild homology \cite{p}, Voronov  on Swiss-Cheese Operad \cite{vo}, and the explicit description of $H^{S^2}_{\bullet}(A,A)$ from \cite{cs} and \cite{la}. 

In the first section we recall  the insertion non-symmetric operad structure on the tensor algebra and exterior algebra. In the second section we introduce  graded Swiss-Cheese (GSD) operads. These are going to give the  right setting to formulate the main construction of this paper.  GSD operads are not really operads, they are somehow similar with Voronov's swiss-cheese operads but with  an extra degree of rigidity. 
As a first interesting example we construct the tensor GSC-operad $T^{S^2}_V$, which is the analog of the tensor algebra. The main result is in the third section where we introduced the exterior GSC-operad  ${\Lambda}^{S^2}_V$. We  state the universality property for this exterior GSC-operad, and show that ${\Lambda}^{S^2}_V(m)=0$ if $m>2dim(V)+1$. Then we give an explicit computation for $dim(V)=2$, in particular we show the existence of a determinant like function for the case $dim(V)=2$.    In the last section we make the  conjecture that ${\Lambda}^{S^2}_V(2dim(V)+1)=1$ for any $V$. Among other things, this would give a determinant like function for vector spaces of finite dimension, and it would open several possible directions of research. 

In the appendix we present some numerical computations for the case $dim(V)=3$. These results were obtained in collaboration with Ana Lorena Gherman.

\section{Preliminaries}

\subsection{Operads}

In this paper $k$ is a field, and $\otimes=\otimes_k$. We recall from \cite{mss} the definition of a non-symmetric operad. The general context in which operads are defined is that of a monoidal category $\mathcal{C}$ (see \cite{lo2}, \cite{mss}). In this paper the category $\mathcal{C}$  will be the category of vector spaces over $k$.   
\begin{definition}\label{opd} A non-symmetric unital operad is a sequence of objects ${\mathcal P}=\{{\mathcal P}(n)\}_{n\geq 1}\in k-mod$ together with linear maps
$$\circ_i:{\mathcal P}(m)\otimes {\mathcal P}(n)\to {\mathcal P}({m+n-1}),$$
one for each $m$, $n\geq 1$, each $1\leq i\leq m$, and a distinguished element $\mathbf{1}\in {\mathcal P}(1)$ such that the following relations hold for all $x\in {\mathcal P}(m)$, $y\in {\mathcal P}(n)$ and $z\in {\mathcal P}(p)$
\begin{eqnarray}
&(x\circ_j z)\circ_i y=(x\circ_i y)\circ_{n+j-1}z,\; {\rm if} \; 1\leq i<j\leq m,&\\
&(x\circ_i y)\circ_{i+j-1}z=x\circ_i(y\circ_j z),\; {\rm if} \; 1\leq i\leq m\; {\rm and} \; 1\leq j\leq n,&\\
&x\circ_i \mathbf{1}= x, \; {\rm if} \; 1\leq i\leq m, &\\
&\mathbf{1}\circ_1 x=x.&
\label{operad}
\end{eqnarray}
\end{definition}

\begin{definition} Let ${\mathcal P}=({\mathcal P}(n), \circ_i, 1)$ be an operad, and ${\mathcal I}=\{{\mathcal I}(n)\}_{n\geq 1}$ a collection of $k$-subspace of $\{{\mathcal P}(n)\}_{n\geq 1}$. We say that ${\mathcal I}$ is a ideal in ${\mathcal P}$ if $x\circ_i a\in {\mathcal I}(m+n-1)$, and $a\circ_j x\in {\mathcal I}(m+n-1)$ for all $x\in {\mathcal P}(n)$, $a\in {\mathcal I}(m)$, $1\leq i\leq n$, and $1\leq j\leq m$. 

\end{definition}

\subsection{The Operad Structure on $T(V)$}

Take $V$ a $k$-vector space. We denote by $T(V)$ the tensor algebra $T(V)=k\oplus V\oplus V^{\otimes 2} \oplus V^{\otimes 3}\oplus ...$
with the usual product. We have an non-symmetric operad structure on $T(V)$ as follows. 
\begin{example} \label{ex1}
For all $n\geq 1$ we take  $P(n)=V^{\otimes n-1}$. For $m,n \in \mathbb{N}$,  and for all $1\leq i\leq m$ we define $\circ_i:P(m)\otimes P(n)\to P(m+n-1)$ determined by 
\begin{eqnarray}
(v_1\otimes ...\otimes v_{m-1})\circ_i (w_1\otimes ...\otimes w_{n-1})=v_1\otimes ...\otimes v_{i-1}\otimes w_1\otimes ...\otimes w_{n-1}\otimes v_{i}\otimes ...\otimes v_{m-1}.
\end{eqnarray}
With the distinguished element $1\in k=P(1)$. One can easily check that the conditions from Definition \ref{operad} are satisfied, and so we get an operad structure on $T(V)$. Notice that the grading is shifted by $1$ compared with the usual grading of $T(V)$. 
\end{example}

\begin{example}
Take $I$ to be the ideal in $T(V)$ generated by $u\otimes v-v\otimes u$ for all $u, v\in V$. One can check that $I$ is an ideal in the operad $P(n)$ from Example \ref{ex1}. In particular we have an operad structure on the symmetric algebra $S(V)$.  Notice that in $S(V)$ the operation $\circ_i$ is nothing but the usual multiplication in $S(V)$.  
\end{example}

\begin{example} 
Take  $J$ to be the ideal in $T(V)$ generated by $u\otimes v+v\otimes u$ for all $u,v\in V$. One can check that $J$ is an ideal in the operad $P(n)$ from Example \ref{ex1}. In particular we have an operad structure on the exterior algebra $\Lambda(V)$.  In $\Lambda(V)$ the operation $\circ_i$ is 
\begin{eqnarray}
(v_1\wedge ...\wedge v_{m-1})\circ_i (w_1\wedge ...\wedge w_{n-1})=(-1)^{(n-1)(m-i)} (v_1\wedge ...\wedge v_{m-1}\wedge w_1\wedge ...\wedge w_{n-1}).
\end{eqnarray}\label{example3}
\end{example}

\section{Graded Swiss-Cheese Operads}

\subsection{GSC-Operads}
In this section we introduce GSD operads, which will provide the correct setting to generalize the exterior algebra of a vector space.  The  definition  is inspired by  Voronov's  Swiss-Cheese operad (\cite{vo}), results on higher Hochschild cohomology over $S^2$ (\cite{cs}, \cite{la}), and secondary Hochschild cohomology (\cite{sta}, \cite{s3}). First we need to define what we mean by a bioperad. 

\begin{definition} Suppose that for all $n, m\geq 1$ we have a vector space ${\mathcal B}(m,n)$ such that for every fixed $n$ there is an operad $({\mathcal B}(m,n), \circ^n_i)$, and for every fixed $m$ there is an operad $({\mathcal B}(m,n), \bullet^m_i)$, where $$\circ^n_i: {\mathcal B}(m,n) \otimes {\mathcal B}(p,n) \to {\mathcal B}(m+p-1,n),$$ 
 $$\bullet^m_i: {\mathcal B}(m,n)\otimes {\mathcal B}(m,q) \to {\mathcal B}(m,n+q-1).$$  
We say that $({\mathcal B}(m,n),\circ^n_i, \bullet^m_i)$ is a bioperad if for all $A\in {\mathcal B}(m,n)$, $B\in {\mathcal B}(m,q)$, $C\in {\mathcal B}(p,n)$, and  $D\in {\mathcal B}(p,q)$ we have 
$$(A\circ^n_i C)\bullet^{m+p}_j(B\circ^q_i D)= (A\bullet^m_j B)\circ^{n+q}_i(C\bullet^p_j D).$$
We say that $({\mathcal B}(m,n),\circ^n_i, \bullet^m_i)$ is symmetric if for every $m$, $n$ there exists a map $tr: {\mathcal B}(m,n) \to {\mathcal B}(n,m)$ such that $$(A^{tr})^{tr}=A$$ and 
$$(A\circ^n_i C)^{tr}=A^{tr}\bullet^n_i C^{tr}$$
for all $A\in {\mathcal B}(m,n)$, and $C\in {\mathcal B}(p,n)$. 
\end{definition}

In what follows it will be convenient to use the upper triangular tensor notation for the the tensor product of three vector spaces $U$, $V$ and $W$ namely $$\otimes\begin{pmatrix} 
U & W\\
 & V 
\end{pmatrix}.$$ 
See \cite{sta} and \cite{s3} for similar situations where this notation is very useful to keep track of certain operations.

\begin{definition} Let $({\mathcal B}(m,n),\circ^n_i, \bullet^m_i)$ be a symmetric bioperad. Suppose that for all $m\geq 1$ we have a vector space ${\mathcal T}(m)$, and for $m, n\geq 1$ and all $1\leq i\leq m$ we have an operation 
$$\diamond_i:\otimes\begin{pmatrix} 
{\mathcal T}(m) & {\mathcal B}(m,n)\\
 & {\mathcal T}(n) 
\end{pmatrix}\to {\mathcal T}(m+n-1),$$ where for $x\in {\mathcal T}(m)$, $y\in {\mathcal T}(n)$  and $A\in {\mathcal B}(m,n)$ we denote the evaluation by  $\begin{pmatrix} 
x & A\\
 \diamond_i& y 
\end{pmatrix}\in {\mathcal T}(m+n-1).$ 

We say that $({\mathcal T}(m), \diamond_i)$ is a Graded Swiss-Cheese (GSC) operad over $({\mathcal B}(m,n),\circ^n_i, \bullet^m_i)$ if for every $x\in {\mathcal T}(m)$, $y\in {\mathcal T}(n)$, $z\in {\mathcal T}(p)$, $A\in {\mathcal B}(m,n)$, $B\in {\mathcal B}(m,p)$,  and $C\in {\mathcal B}(n,p)$ we have:

$$\begin{pmatrix} 
\begin{pmatrix} 
x & B\\
 \diamond_j& z
\end{pmatrix} & A\circ^n_j C^{tr}\\
 \diamond_i& y 
\end{pmatrix}= \begin{pmatrix} 
\begin{pmatrix} 
x & A\\
 \diamond_i& y 
\end{pmatrix} & B\circ^{p}_{i} C\\
 \diamond_{j+n-1}& z 
\end{pmatrix} {\rm ~for} ~1\leq i<j\leq m,    $$
\\

$$\begin{pmatrix} 
\begin{pmatrix} 
x & A\\
 \diamond_i& y 
\end{pmatrix} & B\circ^p_i C\\
 \diamond_{j+i-1}& z 
\end{pmatrix}= 
\begin{pmatrix} 
 x& A\bullet^{m}_{j} B\\
 \diamond_{i}& \begin{pmatrix} 
y & C\\
 \diamond_{j}& z 
\end{pmatrix}
\end{pmatrix} {\rm ~for} ~1\leq i\leq m, ~ 1\leq j \leq n,  $$
Moreover, there exists an element ${\bf 1} \in {\mathcal T}(1)$ such that 
$$\begin{pmatrix} 
x & 1_{n,0}\\
 \diamond_i& {\bf 1}
\end{pmatrix}=x {\rm ~for} ~1\leq i\leq m,$$

$$\begin{pmatrix} 
{\bf 1} & 1_{0,n}\\
 \diamond_1& x
\end{pmatrix}=x.$$

\end{definition} 

\begin{example}
Any non symmetric operad $({\mathcal B}, \diamond_i)$ is a GSC operad, if we take ${\mathcal B}(m,n)=k$ for all $m$ and $n$, with the trivial bi-operad structure.  
\end{example}

\begin{definition} Let $({\mathcal T}(m), \diamond_i)$ be a GSC operad over $({\mathcal B}(m,n),\circ^n_i, \bullet^m_i)$, and $\mathcal{I}(m)$ subspaces in ${\mathcal T}(m)$ for all $m\geq 1$. We say that $\{\mathcal{I}(m)\}_{m\geq 1}$ is an ideal of ${\mathcal T}$ if for every $x\in \mathcal{I}(m)$, $y\in {\mathcal T}(n)$, and $A\in {\mathcal B}(m,n)$ we have $\begin{pmatrix} 
x & A\\
 \diamond_i& y 
\end{pmatrix}\in {\mathcal I}(m+n-1)$ for all $1\leq i\leq m$,  and $\begin{pmatrix} 
y & A^{tr}\\
 \diamond_i& x 
\end{pmatrix}\in {\mathcal I}(m+n-1)$ for all $1\leq i\leq n$.  
\end{definition}

\begin{remark} Let $({\mathcal T}(m), \diamond_i)$ be a GSC operad over $({\mathcal B}(m,n),\circ^n_i, \bullet^m_i)$, and  $\{\mathcal{I}(m)\}_{m\geq 1}$ an ideal of ${\mathcal T}$. Then one can define the quotient GSC operad ${\mathcal T}/{\mathcal I}$. The conditions in the definition of an ideal imply  that all the maps are well defined. The other identities follow because  ${\mathcal T}$ is a GSC operad. Just like for algebras, one can talk about the ideal generated by a subset.
\end{remark}

\subsection{The tensor GSC operad $T^{S^2}_V$}

Next we  give the first interesting example of a GSC-operad structure. One can think that this example is modeled by the sphere $S^2$, just like the operad in Example \ref{ex1} is modeled to the sphere $S^1$.

First we  introduce a bioperad associated to a space $V$. This will be a direct  generalization of Example \ref{ex1}. In order to keep track of the bi-grading we will use the tensor matrix notations introduced in \cite{cs}. 

 Let $V$ be a vector space. For all $m, n\geq 1$ we define ${\mathcal B}_V(m,n)= V^{\otimes (m-1)(n-1)}$. A generic simple tensor in ${\mathcal B}_V(m,n)$ is a tensor matrix with $m-1$ rows and $(n-1)$ columns. 
$$A=\otimes_{k,l} (u_{k,l})=\otimes \begin{pmatrix} 
u_{1,1}&u_{1,2}&...&u_{1,n-1}\\
u_{2,1}&u_{2,2}&...&u_{2,n-1}\\
.&.&...&.\\
u_{m-1,1}&u_{m-1,2}&...&u_{m-1,n-1}\\
\end{pmatrix}\in {\mathcal B}_V(m,n),$$
where $u_{k,l}\in V$. A general element in in ${\mathcal B}_V(m,n)$ is a sum of such simple tensors.  

We need to define $\circ^n_i: {\mathcal B}_V(m,n) \otimes {\mathcal B}_V(p,n) \to {\mathcal B}_V(m+p-1,n),$  and $\bullet^m_i: {\mathcal B}_V(m,n)\otimes {\mathcal B}_V(m,q) \to {\mathcal B}_V(m,n+q-1).$ For 
$A=\otimes_{k,l} (u_{k,l})\in {\mathcal B}_V(m,n)$,
 $C=\otimes_{k,l} (v_{k,l})\in {\mathcal B}_V(p,n)$,  and $1\leq i\leq m$ we define 
$$A\circ_i^n C=\otimes \begin{pmatrix} 
u_{1,1}&u_{1,2}&...&u_{1,n-1}\\
.&.&...&.\\
u_{i-1,1}&u_{i-1,2}&...&u_{i-1,n-1}\\
v_{1,1}&v_{1,2}&...&v_{1,n-1}\\
.&.&...&.\\
v_{p,1}&v_{p,2}&...&v_{p,n-1}\\
u_{i,1}&u_{i,2}&...&u_{i,n-1}\\
.&.&...&.\\
u_{m-1,1}&u_{m-1,2}&...&u_{m-1,n-1}\\
\end{pmatrix}\in {\mathcal B}_V(m+p-1,n).$$
Next for $A=\otimes_{k,l} (u_{k,l})\in {\mathcal B}_V(m,n)$, $B=\otimes_{k,l} (w_{k,l})\in {\mathcal B}_V(m,q)$,  and $1\leq i\leq n$ we define 
$$A\bullet_i^m B=\otimes \begin{pmatrix} 
u_{1,1}&...&u_{1,i-1}&w_{1,1}&...&w_{1,q-1}&u_{1,i}&...&u_{1,n-1}\\
u_{2,1}&...&u_{2,i-1}&w_{2,1}&...&w_{2,q-1}&u_{2,i}&...&u_{2,n-1}\\
.&...&.&.&...&.&.&...&.\\
u_{m-1,1}&...&u_{m-1,i-1}&w_{m-1,1}&...&w_{m-1,q-1}&u_{m-1,i}&...&u_{m-1,n-1}\\
\end{pmatrix}\in {\mathcal B}_V(m,n+q-1).$$

Finally   for $A=\otimes_{k,l} (u_{k,l})\in {\mathcal B}_V(m,n)$ we define  $tr(A)=tr(\otimes_{k,l} (u_{k,l}))=\otimes_{l,k} (u_{l,k})\in {\mathcal B}_V(n,m)$ (i.e. the transpose of the tensor matrix).

\begin{example}  The above structures define  a bioperad denoted $({\mathcal B}_V(m,n), \circ_i^n, \bullet_i^m)$.  
\end{example}

Let $V$ be a $k$-vector space. For $n\geq 0$ we take 
$${\mathcal T}^{S^2}_V(n+1)=V^{\otimes \frac{n(n-1)}{2}}=\otimes\begin{pmatrix} 
k & V & \cdots & V&V\\
 & k & \cdots & V&V\\
 & &\ddots & \vdots & \vdots\\
 & & & k & V\\
 & & & & k\end{pmatrix},$$
where the tensor matrix is of dimension $n\times n$. Notice that when $n=0$, or $1$ we get a one dimensional $k$-vector space. To distinguish between these two cases we use the convention that ${\mathcal T}^{S^2}_V(1)=k{\bf 1} $ and   ${\mathcal T}^{S^2}_V(2)=k(1)$. 
For $m\geq 3$ a generic simple tensor in ${\mathcal T}^{S^2}_V(m)$ is a  tensor upper triangular matrix 
$$x=\otimes_{k<l} (u_{k,l})=\otimes\begin{pmatrix} 
1 & u_{1,2} &u_{1,3}& \cdots & u_{1,m-2}&u_{1,m-1}\\
 & 1 & u_{2,3}&\cdots & u_{2,m-2}&u_{2,m-1}\\
 &  & 1&\cdots & u_{3,m-2}&u_{3,m-1}\\
 & &&\ddots & \vdots & \vdots\\
 & & && 1 & u_{m-2,m-1}\\
 & & && & 1\end{pmatrix}.$$
A general element in ${\mathcal T}^{S^2}_V(m)$ is a sum of such simple tensors. 

For $1\leq i\leq m$ we define  $\diamond_i:\otimes\begin{pmatrix} 
{\mathcal T}^{S^2}_V(m) & {\mathcal B}_V(m,n)\\
 & {\mathcal T}^{S^2}_V(n) 
\end{pmatrix}\to {\mathcal T}^{S^2}_V(m+n-1),$
as follows. If $x=\otimes_{k<l} (u_{k,l})\in {\mathcal T}^{S^2}_V(m)$ , $y=\otimes_{k<l} (v_{k,l})\in {\mathcal T}^{S^2}_V(n)$, and $A=\otimes_{k,l} (w_{k,l})\in {\mathcal B}_V(m,n)$ then 
\setcounter{MaxMatrixCols}{20}
$$\begin{pmatrix} 
x & A\\
 \diamond_i& y 
\end{pmatrix}=\otimes\begin{pmatrix}
1 & u_{1,2} &\cdots&u_{1,i-1}&w_{1,1}&w_{1,2}&\cdots &
w_{1,n-1}&u_{1,i}& \cdots & u_{1,m-1}\\
 & 1 &\cdots&u_{2,i-1}&w_{2,1}&w_{2,2}&\cdots &w_{2,n-1}&u_{2,i}&\cdots &u_{2,m-1}\\
 &&  &.&.&.&\cdots &.&.&\cdots &.\\
 & & & u_{i-2,i-1}&w_{i-2,1}&w_{i-2,2}&\cdots &w_{i-2,n-1}&u_{i-2,i}&\cdots &u_{i-2,m-1}\\
 & & & 1&w_{i-1,1}&w_{i-1,2}&\cdots &w_{i-1,n-1}&u_{i-1,i}&\cdots &u_{i-1,m-1}\\
 &  & & &1&v_{1,2}&\cdots &v_{1,n-1}&w_{i,1}&\cdots &w_{m-1,1}\\
 &  & & &&1&\cdots &v_{2,n-1}&w_{i,2}&\cdots &w_{m-1,2}\\
&  & & & & &\cdots &.&. &\cdots &.\\
&  & & & & & &v_{n-2,n-1}&w_{i,n-2}&\cdots &w_{m-1,n-2}\\
&  & & & & & &1&w_{i,n-1}&\cdots &w_{m-1,n-1}\\
 &  & & & & & &  &1&\cdots &u_{i,m-1}\\
  &  & & & & & & & &.&.\\
  &  & & & & & & &  &1&u_{m-1,m-1}\\
	  &  & & & & & & & & &1
 \end{pmatrix}.  $$
Or if we use the compact representation for the tensor matrices with $x=\otimes\begin{pmatrix}
 X_{1}^{i-1}& X_{1,i}^{i-1,m-i}\\
 &X_{i}^{m-i}
 \end{pmatrix} \in {\mathcal T}^{S^2}_V(m)$, $y\in {\mathcal T}^{S^2}_V(n)$, and $A=\otimes\begin{pmatrix}
 A_{1,1}^{i-1,n-1} \\
A_{i,1}^{m-i, n-1}
 \end{pmatrix}\in {\mathcal B}_V(m,n)$, then 
$$\begin{pmatrix} 
x & A\\
 \diamond_i& y 
\end{pmatrix}=
\otimes\begin{pmatrix}
 X_{1}^{i-1}&A_{1,1}^{i-1,n-1}& X_{1,i}^{i-1,m-i}\\
&y& tr(A_{i,1}^{m-i, n-1})\\
 &&X_{i}^{m-i}
 \end{pmatrix}.$$
To summarize we have the following result.
\begin{example}
$({\mathcal T}^{S^2}_V(m), \diamond_i)$ is a GSC operad over $({\mathcal B}_V(m,n), \circ_i^n, \bullet_i^m)$. \label{ts2}
\end{example} 

\begin{remark} Starting with a pair of vector spaces $U$ and $V$,  one can define a GSC operad by taking 
$${\mathcal T}^{(S^1,D_2)}_{U,V}(n+1)=U^{\otimes n}\otimes V^{\otimes \frac{n(n-1)}{2}}=\otimes\begin{pmatrix} 
U & V & \cdots & V&V\\
 & U & \cdots & V&V\\
 & &\ddots & \vdots & \vdots\\
 & & & U & V\\
 & & & & U\end{pmatrix},$$
where the tensor matrix is of dimension $n\times n$. The $\diamond_i$ operations are similar with the ones from Example \ref{ts2} (see also \cite{s3}).  
Since in this paper we will not explore this example, we leave the details to the reader. For the simplicial model that justifies  the notation ${\mathcal T}^{(S^1,D_2)}_{U,V}$ see \cite{bm}. 
\end{remark}

\section{The Exterior GSC-operad ${\Lambda}^{S^2}_V$}
In this section  we introduce GSC operad that behave  a lot like the exterior algebra $\Lambda(V)$. We study some of its properties, and give explicit computations when $dim(V)=2$. 

\begin{definition} Let $V$ be a vector space, and $({\mathcal T}^{S^2}_V(m), \diamond_i)$ GSC-operad from Example \ref{ts2}. Consider the ideal ${\mathcal E}^{S^2}_V$ generated by the elements $\otimes \begin{pmatrix} 
1& v&v\\
&1&v\\
& &1
\end{pmatrix}$ for all $v\in V$. Define the exterior GSC operad ${\Lambda}^{S^2}_V$ over $({\mathcal B}_V(m,n), \circ_i^n, \bullet_i^m)$ as the quotient 
${\mathcal T}^{S^2}_V/{\mathcal E}^{S^2}_V$. 
\end{definition} 
\begin{remark}
Notice that ${\Lambda}^{S^2}_V(m)=0$ if $m>R(3,3,...,3)$ (where $R(3,3,...,3)$ is the multicolor Ramsey number with $dim(V)$ colors). 
\end{remark}

To  distinguish among elements in $V^{\frac{n(n-1)}{2}}$, ${\mathcal T}^{S^2}_V(n+1)$ and ${\Lambda}^{S^2}_V(n+1)$ we will use the following notations respectively 
$$\begin{pmatrix} 
0& v_{1,2}&...&v_{1,n-1}&v_{1,n}\\
& 0&...&v_{2,n-1}&v_{2,n}\\
& &...&.&.\\
& & &0&v_{n-1,n}\\
& & & &0
\end{pmatrix}\in V^{\frac{n(n-1)}{2}},$$

$$\otimes \begin{pmatrix} 
1& v_{1,2}&...&v_{1,n-1}&v_{1,n}\\
& 1&...&v_{2,n-1}&v_{2,n}\\
& &...&.&.\\
& & &1&v_{n-1,n}\\
& & & &1
\end{pmatrix}\in {\mathcal T}^{S^2}_V(n+1),$$ 

and $$\begin{pmatrix} 
1& v_{1,2}&...&v_{1,n-1}&v_{1,n}\\
& 1&...&v_{2,n-1}&v_{2,n}\\
& &...&.&.\\
& & &1&v_{n-1,n}\\
\wedge& & & &1
\end{pmatrix}\in {\Lambda}^{S^2}_V(n+1).$$

\begin{definition} Let $V$ and $W$ be two vector spaces, and $\phi:V^{\frac{n(n-1)}{2}}\to W$ a $k$-multilinear map. We say that $\phi$ is $2$-alternating if 
$\phi\begin{pmatrix} 
0& v_{1,2}&...&v_{1,n-1}&v_{1,n}\\
& 0&...&v_{2,n-1}&v_{2,n}\\
& &...&.&.\\
& & &0&v_{n-1,n}\\
& & & &0
\end{pmatrix}=0$ whenever $v_{i,j}=v_{i,k}=v_{j,k}$ for some $i<j<k$. 
\end{definition}
Just like for the exterior algebra we have the following universality property for the exterior operad ${\Lambda}^{S^2}_V$. 
\begin{proposition}
Let $V$ and $W$ be two vector spaces, and $\phi:V^{\frac{n(n-1)}{2}}\to W$ a $2$-alternating $k$-multilinear map. Then there exists an unique $k$-linear map $f: {\Lambda}^{S^2}_V(n+1)\to W$ such that $f\alpha=\phi$, where $\alpha$ is the natural map from $V^{\frac{n(n-1)}{2}}$ to ${\Lambda}^{S^2}_V(n+1)$.\label{2alt}
\end{proposition}
\begin{proof}
It follows directly from definitions. 
\end{proof}

\begin{lemma} Assume that $char(k)$ is not $2$ or $3$, then the ideal ${\mathcal E}^{S^2}_V$ is generated by any of the following sets: \label{lemma3}\\
1) $\{ \otimes \begin{pmatrix} 
1& v&v\\
&1&v\\
& &1
\end{pmatrix}\vert {\rm ~ for ~ all} ~v\in V\}$.\\
2) $\{ \otimes\begin{pmatrix} 
1& v&v\\
&1&u\\
& &1
\end{pmatrix}+\otimes\begin{pmatrix} 
1& v&u\\
&1&v\\
& &1
\end{pmatrix}+\otimes\begin{pmatrix} 
1& u&v\\
&1&v\\
& &1
\end{pmatrix}\vert {\rm ~ for ~ all} ~u,~v\in V\}$.\\
3) $\{ \otimes\begin{pmatrix} 
1& u&v\\
&1&w\\
& &1
\end{pmatrix}+\otimes\begin{pmatrix} 
1& u&w\\
&1&v\\
& &1
\end{pmatrix}+\otimes\begin{pmatrix} 
1& v&u\\
&1&w\\
& &1
\end{pmatrix}+ \otimes\begin{pmatrix} 
1& w&u\\
&1&v\\
& &1
\end{pmatrix}+\otimes\begin{pmatrix} 
1& v&w\\
&1&u\\
& &1
\end{pmatrix}+\otimes\begin{pmatrix} 
1& w&v\\
&1&u\\
& &1
\end{pmatrix} \vert \\{\rm ~ for ~ all} ~u, ~v, ~w\in V\}$.\\
\end{lemma}
\begin{proof}
This is a straightforward computation. For example, to prove that the elements in $3)$ are in ${\mathcal E}^{S^2}_V$ consider the vector $\otimes \begin{pmatrix} 
1& u+v+w&u+v+w\\
&1&u+v+w\\
& &1
\end{pmatrix}\in {\mathcal E}^{S^2}_V$, and use the linearity of the tensor product. 
\end{proof}

\begin{lemma} Let $V$ be a vector space, $\omega \in V$, and $n\geq 3$. Take  $X=\begin{pmatrix} 
0& v_{1,2}&...&v_{1,n-1}&v_{1,n}\\
& 0&...&v_{2,n-1}&v_{2,n}\\
& &...&.&.\\
& & &0&v_{n-1,n}\\
& & & &0
\end{pmatrix}\in V^{\frac{n(n-1)}{2}}$ such that $\omega$ appears at least $n$ times amongst  the elements in the set $\{v_{i,j}\;\vert \; 1\leq i<j\leq n\}$. Then the image of $X$ in ${\Lambda}^{S^2}_V(n+1)$  is trivial, that is   $\hat{X}=\begin{pmatrix} 
1& v_{1,2}&...&v_{1,n-1}&v_{1,n}\\
& 1&...&v_{2,n-1}&v_{2,n}\\
& &...&.&.\\
& & &1&v_{n-1,n}\\
\wedge& & & &1
\end{pmatrix}=0\in {\Lambda}^{S^2}_V(n+1)$. \label{lemma5}
\end{lemma}
\begin{proof} We will prove this result by induction. If $n=3$ then the image of $X$ is 
$$\hat{X}=\begin{pmatrix} 
1& v_{1,2}&v_{1,3}\\
& 1&v_{2,3}\\
\wedge& &1\\
\end{pmatrix}=\begin{pmatrix} 
1& \omega&\omega\\
& 1&\omega\\
\wedge& &1\\
\end{pmatrix},$$
 which is obvious $0$ in ${\Lambda}^{S^2}_V(4)$. 

Take $n\geq 4$. First assume that the last column of $X$ has at least two entries that are equal to $\omega$, say $v_{i,n}=v_{j,n}=\omega$ for some $i<j$. If $v_{i,j}=\omega$ then obviously we have that $\hat{X}=0$. If $v_{i,j}\neq \omega$ then we can use the second relation from Lemma \ref{lemma3} to get 
$$\hat{X}=-\hat{X_1}-\hat{X_2},$$
where $X_{1}$ and $X_2$ are obtained from $X$ by interchanging the elements $v_{i,j}$ with $v_{i,n}$, respectively $v_{j,n}$ (i.e. we use the relation 
$\otimes\begin{pmatrix} 
1& v_{i,j}&\omega\\
&1&\omega\\
& &1
\end{pmatrix}=-\otimes\begin{pmatrix} 
1& \omega&v_{i,j}\\
&1&\omega\\
& &1
\end{pmatrix}-\otimes\begin{pmatrix} 
1& \omega&\omega\\
&1&v_{i,j}\\
& &1
\end{pmatrix}$).

Notice that the new upper triangular matrices $X_1$ and $X_2$ have the same entries as $X$, but one fewer $\omega$'s in the last column.  If we show that $\hat{X_1}$ and $\hat{X_2}$ are $0$ then we are done. 

By the above argument, we may assume  that $X$ has at most one entry equal to $\omega$ in the last column. This means that $X$ has at least $n-1$ entry equal to $\omega$ in first $n-1$ columns. Denote 
$$Y=\begin{pmatrix} 
0& v_{1,2}&...&v_{1,n-2}&v_{1,n-1}\\
& 0&...&v_{2,n-2}&v_{2,n-1}\\
& &...&.&.\\
& & &0&v_{n-2,n-1}\\
& & & &0
\end{pmatrix}\in V^{\frac{(n-1)(n-2)}{2}},$$
by the induction hypothesis, we have that $\hat{Y}=0\in {\Lambda}^{S^2}_V(n)$. Finally, notice that $\hat{X}=\begin{pmatrix} 
\hat{Y} & A\\
 \diamond_{n-1}& (1)
\end{pmatrix}$ where $(1)\in {\Lambda}^{S^2}_V(2)$, and $A\in {\mathcal B}_V(n-1,1)$ is the tensor corresponding to the last column in $X$. This proves that $\hat{X}=0$. 
\end{proof}

\begin{proposition} Let $V$ be a vector space with $dim(V)=d$. Then ${\Lambda}^{S^2}_V(n+1)=0$ for all $n>2d$. \label{propdim}
\end{proposition}
\begin{proof} Let $\mathcal{E}=\{ e_1, ...,e_d\}$ be a basis for $V$. It is obvious that 
$$\mathcal{G}(n+1)=\{\begin{pmatrix} 
1& v_{1,2}&...&v_{1,n-1}&v_{1,n}\\
& 1&...&v_{2,n-1}&v_{2,n}\\
& &...&.&.\\
& & &1&v_{n-1,n}\\
\wedge& & & &1
\end{pmatrix}\; \vert \; v_{i,j}\in \mathcal{E}\},$$ is a system of generators for ${\Lambda}^{S^2}_V(n+1)$. To prove our statement it is enough to show that every element in $\mathcal{G}(n+1)$ is trivial. 

The number of entries in the upper triangular matrix $X=\begin{pmatrix} 
0& v_{1,2}&...&v_{1,n-1}&v_{1,n}\\
& 0&...&v_{2,n-1}&v_{2,n}\\
& &...&.&.\\
& & &0&v_{n-1,n}\\
& & & &0
\end{pmatrix}$ is equal to $\frac{n(n-1)}{2}$. On the other hand, by Lemma \ref{lemma5} we know that in order to get a nonzero element we cannot have more then $n-1$ copies of each of the $d$ elements in $\mathcal{E}$. In other words if we want an nonzero element in $\mathcal{G}(n+1)$ then we must have $\frac{n(n-1)}{2}\leq d(n-1)$, or equivalently $n\leq 2d$. 
\end{proof}

\begin{proposition} Suppose that $V$ is a two dimensional vector space with a basis $\{a,b\}$. If $m> 5$ then  ${\Lambda}^{S^2}_V(m)=0$. For $m\leq 5$ we have the following basis: 

${\mathcal B}(1)=\{ {\bf 1}\},$

${\mathcal B}(2)=\{\wedge\begin{pmatrix} 
1
\end{pmatrix}\}$ for ${\Lambda}^{S^2}_V(2)$, 

 ${\mathcal B}(3)=\{  \begin{pmatrix} 
1& a\\
\wedge&1
\end{pmatrix}, ~ \begin{pmatrix} 
1& b\\
\wedge&1
\end{pmatrix} \}$ for ${\Lambda}^{S^2}_V(3)$,  

 ${\mathcal B}(4)=\{  \begin{pmatrix} 
1& a&a\\
&1&b\\
\wedge& &1
\end{pmatrix}, ~
\begin{pmatrix} 
1& a&b\\
&1&a\\
\wedge& &1
\end{pmatrix}, ~
\begin{pmatrix} 
1& b&b\\
&1&a\\
\wedge& &1
\end{pmatrix}, ~
\begin{pmatrix} 
1& a&b\\
&1&b\\
\wedge& &1
\end{pmatrix} \}$ for ${\Lambda}^{S^2}_V(4)$, and

 ${\mathcal B}(5)=\{ \begin{pmatrix} 
1& a&b&b\\
&1&a&b\\
&  &1&a\\
\wedge&&&1
\end{pmatrix} \}$ for ${\Lambda}^{S^2}_V(5)$. 
\label{prop3}
\end{proposition}
\begin{proof} It is a straightforward but tedious computation to show that the above elements  generate the corresponding vector spaces. We need to show that they are also linearly independent. 
This follows the same line of proof as for the  exterior algebra ${\Lambda}(V)$. 

We start by showing  that $ \begin{pmatrix} 
1& a&b&b\\
&1&a&b\\
& &1&a\\
\wedge&&&1
\end{pmatrix}\neq 0 $. Consider the map $Det^{S^2}:V^6\to k$ determined by 

\begin{eqnarray*} &Det^{S^2}\begin{pmatrix} 
0& (\alpha_{1,2},\beta_{1,2})&(\alpha_{1,3},\beta_{1,3})&(\alpha_{1,4},\beta_{1,4})\\
&0&(\alpha_{2,3},\beta_{2,3})&(\alpha_{2,4},\beta_{2,4})\\
& &0&(\alpha_{3,4},\beta_{3,4})\\
&&&0
\end{pmatrix}=&\\
&\alpha_{1,2}\alpha_{2,3}\alpha_{3,4}\beta_{1,3}\beta_{2,4}\beta_{1,4}+
\alpha_{1,2}\beta_{2,3}\alpha_{3,4}\beta_{1,3}\beta_{2,4}\alpha_{1,4}+
\alpha_{1,2}\beta_{2,3}\beta_{3,4}\alpha_{1,3}\alpha_{2,4}\beta_{1,4}+
\beta_{1,2}\beta_{2,3}\alpha_{3,4}\alpha_{1,3}\alpha_{2,4}\beta_{1,4}+&\\
&\beta_{1,2}\alpha_{2,3}\beta_{3,4}\beta_{1,3}\alpha_{2,4}\alpha_{1,4}+
\beta_{1,2}\alpha_{2,3}\beta_{3,4}\alpha_{1,3}\beta_{2,4}\alpha_{1,4}-
\beta_{1,2}\beta_{2,3}\beta_{3,4}\alpha_{1,3}\alpha_{2,4}\alpha_{1,4}-
\beta_{1,2}\alpha_{2,3}\beta_{3,4}\alpha_{1,3}\alpha_{2,4}\beta_{1,4}-&\\
&\beta_{1,2}\alpha_{2,3}\alpha_{3,4}\beta_{1,3}\beta_{2,4}\alpha_{1,4}-
\alpha_{1,2}\alpha_{2,3}\beta_{3,4}\beta_{1,3}\beta_{2,4}\alpha_{1,4}-
\alpha_{1,2}\beta_{2,3}\alpha_{3,4}\alpha_{1,3}\beta_{2,4}\beta_{1,4}-
\alpha_{1,2}\beta_{2,3}\alpha_{3,4}\beta_{1,3}\alpha_{2,4}\beta_{1,4}.&
\end{eqnarray*}

One can see that $Det^{S^2}$ is $k$-linear in all the six variables. Moreover we have 
\begin{eqnarray*} &Det^{S^2}\begin{pmatrix} 
0& (\alpha,\beta)&(\alpha,\beta)&(\alpha_{1,4},\beta_{1,4})\\
&0&(\alpha,\beta)&(\alpha_{2,4},\beta_{2,4})\\
& &0&(\alpha_{3,4},\beta_{3,4})\\
&&&0
\end{pmatrix}=0,
\end{eqnarray*}

\begin{eqnarray*}
Det^{S^2}\begin{pmatrix} 
0& (\alpha,\beta)&(\alpha_{1,3},\beta_{1,3})&(\alpha,\beta)\\
&0&(\alpha_{2,3},\beta_{2,3})&(\alpha,\beta)\\
& &0&(\alpha_{3,4},\beta_{3,4})\\
&&&0
\end{pmatrix}=0,
\end{eqnarray*}

\begin{eqnarray*}
Det^{S^2}\begin{pmatrix} 
0& (\alpha_{1,2},\beta_{1,2})&(\alpha,\beta)&(\alpha,\beta)\\
&0&(\alpha_{2,3},\beta_{2,3})&(\alpha_{2,4},\beta_{2,4})\\
& &0&(\alpha,\beta)\\
&&&0
\end{pmatrix}=0,
\end{eqnarray*}

\begin{eqnarray*}
Det^{S^2}\begin{pmatrix} 
0& (\alpha_{1,2},\beta_{1,2})&(\alpha_{1,3},\beta_{1,3})&(\alpha_{1,4},\beta_{1,4})\\
&0&(\alpha,\beta)&(\alpha,\beta)\\
& &0&(\alpha,\beta)\\
&&&0
\end{pmatrix}=0.
\end{eqnarray*}
In other words $Det^{S^2}$ is $k$-multi-linear and $2$-alternated, and so by Proposition \ref{2alt} we know that there exists a linear map  $det^{S^2}\in Hom_k(\Lambda^{S^2}_V(5), k)$ such that 

$$det^{S^2}\begin{pmatrix} 
1& (\alpha_{1,2},\beta_{1,2})&(\alpha_{1,3},\beta_{1,3})&(\alpha_{1,4},\beta_{1,4})\\
&1&(\alpha_{2,3},\beta_{2,3})&(\alpha_{2,4},\beta_{2,4})\\
& &1&(\alpha_{3,4},\beta_{3,4})\\
\wedge&&&1
\end{pmatrix}=Det^{S^2}\begin{pmatrix} 
0& (\alpha_{1,2},\beta_{1,2})&(\alpha_{1,3},\beta_{1,3})&(\alpha_{1,4},\beta_{1,4})\\
&0&(\alpha_{2,3},\beta_{2,3})&(\alpha_{2,4},\beta_{2,4})\\
& &0&(\alpha_{3,4},\beta_{3,4})\\
&&&0
\end{pmatrix}.$$

Moreover, sice
\begin{eqnarray*}
Det^{S^2}\begin{pmatrix} 
0& (1,0)&(0,1)&(0,1)\\
&0&(1,0)&(0,1)\\
& &0&(1,0)\\
&&&0
\end{pmatrix}=1,
\end{eqnarray*}
we have  that $0\neq det^{S^2}\in Hom_k(\Lambda^{S^2}_V(5), k)$, and so 
$ \begin{pmatrix} 
1& a&b&b\\
&1&a&b\\
& &1&a\\
\wedge&&&1
\end{pmatrix}\neq 0$ as an element in ${\Lambda}^{S^2}_V(5)$. In particular we get that $dim_k({\Lambda}^{S^2}_V(5))=1$.

Next, suppose that $\alpha$, $\beta$, $\gamma$, and $\delta\in k$ such that 
$$x=\alpha \begin{pmatrix} 
1& a&a\\
&1&b\\
\wedge& &1
\end{pmatrix}+\beta 
\begin{pmatrix} 
1& a&b\\
&1&a\\
\wedge& &1
\end{pmatrix}+\gamma
\begin{pmatrix} 
1& b&b\\
&1&a\\
\wedge& &1
\end{pmatrix}+\delta
\begin{pmatrix} 
1& a&b\\
&1&b\\
\wedge& &1
\end{pmatrix}=0\in {\Lambda}^{S^2}_V(4).$$
Take $y=\wedge\begin{pmatrix} 
1
\end{pmatrix}\in {\Lambda}^{S^2}_V(2)$, and $A=\otimes \begin{pmatrix} 
b\\
a\\
b
\end{pmatrix}\in {\mathcal B}_V(4,2)$. Then we have $\begin{pmatrix} 
x & A\\
 \diamond_4& y 
\end{pmatrix}=0$, or equivalently

$$\alpha \begin{pmatrix} 
1& a&a&b\\
&1&b&a\\
& &1&b\\
\wedge&&&1
\end{pmatrix}+\beta 
\begin{pmatrix} 
1& a&\boxed{b}&\boxed{b}\\
&1&a&a\\
& &1&\boxed{b}\\
\wedge&&&1
\end{pmatrix}+\gamma
\begin{pmatrix} 
1& b&\boxed{b}&\boxed{b}\\
&1&a&a\\
& &1&\boxed{b}\\
\wedge&&&1
\end{pmatrix}+\delta
\begin{pmatrix} 
1& a&\boxed{b}&\boxed{b}\\
&1&b&a\\
& &1&\boxed{b}\\
\wedge&&&1
\end{pmatrix}=0,$$
and so $\alpha=0$. Taking $A=\otimes \begin{pmatrix} 
a\\
b\\
b
\end{pmatrix}$ we get that $\beta=0$. Similarly one can show that $\gamma=\delta=0$, and so ${\mathcal B}(4)$ is a basis for  ${\Lambda}^{S^2}_V(4)$.  An analog argument can be used to show that ${\mathcal B}(3)$ is a basis for  ${\Lambda}^{S^2}_V(3)$.
\end{proof}

\begin{remark} Let $V$ be a two dimensional vector space with a basis $\{a,b\}$, and  $T:V\to V$ a $k$-linear map. One can check that $\wedge^{S^2}(T): {\Lambda}^{S^2}_V(5)\to {\Lambda}^{S^2}_V(5)$ is the multiplication by $det(T)^3$. 
\end{remark}

\begin{remark} One should notice that there is some inconsistency between the degree convention for the exterior algebra $\Lambda(V)$ (with nonzero components in degrees from $0$ to $d$), and the degree convention  we use for ${\Lambda}^{S^2}_V$ (with nonzero components in degrees from $1$ to $2d+1$). One can argue that it might be more natural to shift the degree down with one unit (and have nonzero components from degree $0$ to $2d$). The reason we decided to have this notation is because we want our grading to be consistent with the GSC-operad structure on ${\Lambda}^{S^2}_V$ (see Example \ref{example3} for a similar situation).
\end{remark}

\section{Some Remarks}



From Proposition \ref{propdim} we know that if $dim(V)=d$ then ${\Lambda}^{S^2}_V(m)=0$ for all $m>2d+1$.   Also, when $dim(V)=2$ we know from Proposition \ref{prop3} that $dim_k({\Lambda}^{S^2}_V(5))=1$. We make the following conjecture.
\begin{conjecture}  If $dim_k(V)=d$ then $dim_k({\Lambda}^{S^2}_V(2d+1))=1$. 
\end{conjecture}
If this conjecture is true then we would be able  to define a determinant like function $det^{S^2}$ for any vector space $V$ (and not only for vector spaces of dimension $2$). This potentially  could open several directions of research.

Also in the spirit of Proposition \ref{prop3}, it would also be useful to find an explicit base for ${\Lambda}^{S^2}_V(m)$ for any finite dimensional vector space $V$.  
\begin{remark}
Our results do not depend on the fact that $V$ is a vector space, one can give a similar construction for a free module of finite rank over a commutative ring $R$. 
\end{remark}
\begin{remark}
Exterior algebra plays an important role in geometry and differential topology. A $p$-form, is nothing else but an element in $\wedge^p\Omega^1_M$. It is natural to ask if there  there is  an "$S^2$-geometry" based on ${\Lambda}^{S^2}_{\Omega^1_M}$. 
\end{remark}

\begin{remark}
The construction in this paper is bases on a certain simplicial presentation of the sphere  $S^2$. It is natural to ask if there are similar constructions that can be associated to any simplicial set $X$, or at least to the spheres $S^n$. If such constructions are possible, are they  homotopy invariant? (like higher Hochschild  homology in \cite{p}). 
\end{remark}

\section*{Acknowledgment}
We thank Tong Sun for some discussions and help with  Mathlab. 


\appendix

\maketitle

\section{$dim(V)=3$ (Joint work with Ana Lorena Gherman)}
When the  vector space $V$  has dimension $3$ we were able to compute some of the dimensions of  ${\Lambda}^{S^2}_V(n)$  using Mathlab. Here we give  a summary of these computations.

First consider a vector space $V$ of dimension $d$, and recall the notations from Proposition \ref{propdim}. Notice that once we fix a basis $\mathcal{E}=\{ e_1, ...,e_d\}$  for $V$ then there exists a subset in  $\mathcal{G}(n+1)$ that is a basis for ${\Lambda}^{S^2}_V(n+1)$. Moreover  this basis of ${\Lambda}^{S^2}_V(n+1)$ has a partition determined by multiplicities of the elements in $\mathcal{E}$. 

Let $1\leq n\leq 2d$, and $k_i\geq  0$ for all $1\leq i\leq d$ such that $k_1+k_2+...+k_d=\frac{n(n-1)}{2}$. We denote by $E_n^{(k_1,k_2,...,k_d)}$ the maximum number of linearly independent vectors in $\mathcal{G}(n+1)$ with $k_i$ entries of $e_i$ for all $1\leq i\leq d$. For example, Proposition \ref{prop3} gives that when $d=2$ we have  
$$E_2^{(1,0)}=E_2^{(0,1)}=1,~~~E_3^{(2,1)}=E_3^{(1,2)}=2,~~~ {\rm and} ~E_4^{(3,3)}=1.$$ 

It is easy to see that $E_n^{(k_1,k_2,...,k_d)}$ is invariant under any permutation of the indices $(k_1,k_2,...,k_d)$, so without loose of generality, we may assume that $k_1\geq k_2\geq ...\geq k_d\geq 0$. When going from dimension $d$ to $d+1$ one has $E_n^{(k_1,k_2,...,k_d)}=E_n^{(k_1,k_2,...,k_d,0)}$. Also, from Lemma \ref{lemma5} we have that $E_n^{(k_1,k_2,...,k_d)}=0$ when ever $k_1\geq n$. 

Using the relations from Lemma \ref{lemma3} one can write a linear system whose rank is 
$E_n^{(k_1,k_2,...,k_d)}$.   The following result was obtained using Mathlab.  
\begin{lemma} Let $V$ be a vector space with $dim(V)=3$ then we have:\\
1) $E_2^{(1,0,0)}=1$,\\
2) $E_3^{(2,1,0)}=2$, $E_3^{(1,1,1)}=5$, \\
3) $E_4^{(3,3,0)}=1$,  $E_4^{(3,2,1)}=9$, $E_4^{(2,2,2)}=22$,\\
4) $E_5^{(4,4,2)}=6$,  $E_5^{(4,3,3)}=16$.
\end{lemma}

In particular we can compute the dimension of ${\Lambda}^{S^2}_V(m)$. 
\begin{lemma} Let $V$ be a vector space such that $dim(V)=3$ then\\
1) $dim({\Lambda}^{S^2}_V(1))=1$,\\
2) $dim({\Lambda}^{S^2}_V(2))=1$,\\
3) $dim({\Lambda}^{S^2}_V(3))=3$,\\
4) $dim({\Lambda}^{S^2}_V(4))=17=12+5$,\\
5) $dim({\Lambda}^{S^2}_V(5))=79=3+54+22$,\\
6) $dim({\Lambda}^{S^2}_V(6))=66=18+48$.
\label{lemma7}
\end{lemma}

Notice that we have no information  about $E_6^{(5,5,5)}$ (or $dim({\Lambda}^{S^2}_V(7))$). The algorithm we used to find the above results requires a good amount of computational resources. For $n\leq 5$ we were able to compute $E_n^{(k_1,k_2,k_3)}$ on a regular desktop (total running time was less than one hour). However, when one tries to use the same method for $E_6^{(5,5,5)}$ we need to find the rank o matrix with  $\num{12409320}$ rows and $\num{756756}$ columns (that is about 25TB of data). This is way more than a regular computer can handle.

\bibliographystyle{amsalpha}

\begin{thebibliography}{A}



\bibitem
{bm}
B. R. Corrigan-Salter, and M. D. Staic, \textit{ Higher Order and Secondary Hochschild Cohomology},  C. R. Math. Acad. Sci. Paris, {\bf 354} (2016), no. 11, 1049--1054. 



\bibitem
{cs}
S. Carolus,  and M. D. Staic, 
\textit{$G$-Algebra Structure on the Higher Order Hochschild Cohomology $H^*_{S^2}(A,A)$}, arXiv:1804.05096


\bibitem
{f}
G. Floystad, \textit{The Exterior Algebra and Central Notions in Mathematics}, Notices of the AMS (4) {\bf 62} (2015), 364--371.




\bibitem
{hkr}
G. Hochschild, B. Kostant, and  A. Rosenberg, \textit{Differential forms on regular affine algebras}, Transactions AMS, {\bf 102} (1962), 383--408.



\bibitem
{la} J. Laubacher, \textit{Secondary Hochschild and Cyclic (co)homologies}, Ph.D. thesis,  (2017).



\bibitem
{lo} J. L. Loday, \textit{Cyclic Homology}, Springer-Verlag, Grundlehren der mathematischen Wissenschaften, {\bf 301} (1992).


\bibitem
{lo2} J. L. Loday, and B.  Vallette, \textit{Algebraic Operads}, Springer-Verlag, Grundlehren der mathematischen Wissenschaften, {\bf 346} (2012).



\bibitem
{mss} M. Markl, S. Shnider, and J. Stasheff, \textit{Operads in Algebra, Topology and Physiscs}, Mathematical Surveys and Monographs {\bf 96} (2002).


\bibitem
{p} T. Pirashvili, \textit{Hodge decomposition for higher order Hochschild homology}, Ann. Sci. Ecole Norm. Sup., (4) {\bf 33} (2000), 151--179.



\bibitem
{sta} M. D. Staic, \textit{Secondary Hochschild Cohomology}, Algebras and Representation Theory,  {\bf 19} Issue 1 (2016), pp 47--56.



\bibitem
{s3} M. D. Staic and A. Stancu, 
\textit{Operations on the Secondary  Hochschild  Cohomology},  Homology, Homotopy and Applications,  {\bf 17} (2015), 129-146. 



\bibitem
{vo}
 A. A. Voronov, \textit{The Swiss-Cheese Operad},  Contemporary Mathematics, {\bf 239} (1999), 365--373.


\end{thebibliography}

\end{document}